\documentclass[12pt]{article}
\usepackage{amsmath,amssymb,amsthm,comment}

\newtheorem{theorem}{Theorem}[section]
\newtheorem{thmy}{Theorem}

\newtheorem{corollary}[theorem]{Corollary}

\newtheorem{example}[theorem]{Example}

\newcommand{\dd}{\displaystyle }

\def\barr{\begin{array}}
\def\earr{\end{array}}

\title{Counting conjugacy classes and automorphism-conjugacy classes of ZM-groups}
\author{Marius T\u arn\u auceanu}
\date{September 21, 2025}

\begin{document}

\maketitle

\begin{abstract}
In this paper, we count the number of conjugacy and automorphism-conjugacy classes of elements of a ZM-group. The size of a conjugacy class with respect to these two equivalence relations is also counted.
\end{abstract}

{\small
\noindent
{\bf MSC2020\,:} Primary 20D60; Secondary 20D45, 11A25.

\noindent
{\bf Key words\,:} conjugacy class, automorphism-conjugacy class, ZM-group, group action, Burnside's lemma, Menon's identity.}

\section{Introduction}

Given a finite group $G$, we denote by ${\rm Inn}(G)$ and ${\rm Aut}(G)$ the inner automorphism group and the automorphism group of $G$. Then the orbits of $G$ under the natural actions of ${\rm Inn}(G)$ and ${\rm Aut}(G)$ are called the \textit{conjugacy classes} and the \textit{automorphism-conjugacy classes} of $G$. They are very useful in the study of groups, especially in representation theory. The numbers of these orbits are usually denoted by $k(G)$ and $k'(G)$, respectively, and represent two important invariants of $G$. Note that $k(G)$ is equal to the number of complex irreducible characters of $G$.

In what follows, we will determine these numbers and the sizes of conjugacy classes for ZM-groups. Recall that a \textit{ZM-group} is a finite group with all Sylow subgroups cyclic. By \cite{5}, such a group is of type
\begin{equation}
{\rm ZM}(m,n,r)=\langle a, b \mid a^m = b^n = 1,
\hspace{1mm}b^{-1} a b = a^r\rangle, \nonumber
\end{equation}
where the triple $(m,n,r)$ satisfies the conditions
\begin{equation}
(m,n)=(m,r-1)=1 \mbox{ and } r^n
\equiv 1 \hspace{1mm}({\rm mod}\hspace{1mm}m). \nonumber
\end{equation}\newpage \noindent It is clear that $|{\rm ZM}(m,n,r)|=mn$ and $Z({\rm ZM}(m,n,r))=\langle b^d\rangle$, where $d$ is the multiplicative order of
$r$ modulo $m$, i.e.
\begin{equation}
d=o_m(r)={\rm min}\{k\in\mathbb{N}^* \mid r^k\equiv 1 \hspace{1mm}({\rm mod} \hspace{1mm}m)\}.\nonumber
\end{equation}The subgroups of ${\rm ZM}(m,n,r)$ have been completely described
in \cite{2}. Set
\begin{equation}
L=\left\{(m_1,n_1,s)\in\mathbb{N}^3 \hspace{1mm}\mid\hspace{1mm}
m_1|m,\hspace{1mm} n_1|n,\hspace{1mm} s<m_1,\hspace{1mm}
m_1|s\frac{r^n-1}{r^{n_1}-1}\right\}.\nonumber
\end{equation}Then there is a bijection between $L$ and the subgroup lattice
$L({\rm ZM}(m,n,r))$ of ${\rm ZM}(m,n,r)$, namely the function
that maps a triple $(m_1,n_1,s)\in L$ to the subgroup $H_{(m_1,n_1,s)}$ defined by
\begin{equation}
H_{(m_1,n_1,s)}=\bigcup_{k=1}^{\frac{n}{n_1}}(b^{n_1}a^s)^k\langle a^{m_1}\rangle=\langle a^{m_1},b^{n_1}a^s\rangle.\nonumber
\end{equation}Note that:
\begin{itemize}
\item[-] $|H_{(m_1,n_1,s)}|=\frac{mn}{m_1n_1}$\,, for any $s$ satisfying $(m_1,n_1,s)\in L$;
\item[-] $H_{(m_1,n_1,s)}$ is normal in ${\rm ZM}(m,n,r)$ if and only if $m_1\mid r^{n_1}-1$ and $s=0$ (see \cite{10});
\item[-] $H_{(m_1,n_1,s)}$ is cyclic if and only if $\frac{m}{m_1}\mid r^{n_1}-1$ (see \cite{11});
\item[-] two subgroups of ${\rm ZM}(m,n,r)$ are conjugate if and only if they have the same order (see \cite{1}).
\end{itemize}

Our approach uses some basic tools of group action theory, such as the Orbit-Stabilizer theorem and the Burnside's lemma, and involves modular arithmetic.

\bigskip\noindent{\bf Orbit-Stabilizer theorem.}
{\it Let $G$ be a finite group acting on a finite set $X$ and $x\in X$. Denote by
\begin{equation}
O_x=\{g\circ x \mid g\in G\} \mbox{ and } Stab_G(x)=\{g\in G \mid g\circ x=x\}\nonumber
\end{equation}the orbit and the stabilizer of $x$ in $G$. Then 
\begin{equation}
|O_x|=[G:Stab_G(x)]=\frac{|G|}{|Stab_G(x)|}\,.\nonumber
\end{equation}}

\bigskip\noindent{\bf Burnside's lemma.}
{\it Let $G$ be a finite group acting on a finite set $X$ and
\begin{equation}
Fix(g)=\{x\in X \mid g\circ x=x\}, \mbox{ for } g\in G.\nonumber
\end{equation}Then the number of distinct orbits is
\begin{equation}
\frac{1}{|G|}\sum_{g\in G}|Fix(g)|.\nonumber
\end{equation}}
 
Most of our notations are standard and will not be repeated here. Basic definitions and results on group theory can be found in \cite{5,9}. For number theoretic notions we refer the reader to \cite{6,8}.

\section{Main results}

Throughout this section we will use $G$ to denote ${\rm ZM}(m,n,r)$. The general form of automorphisms of a metacyclic group had been determined in the main result of \cite{3}. For ZM-groups, in \cite{4} we obtained the following:

\begin{theorem}
Each automorphism of ${\rm ZM}(m,n,r)$ is given by
\begin{equation}
b^ua^v\mapsto b^{yu}a^{x_1v+x_2[u]_r},\, u,v\geq 0,\nonumber
\end{equation}for a unique triple of integers $(x_1,x_2,y)$ such that
\begin{equation}
0\leq x_1,x_2<m,\, (x_1,m)=1,\, 0\leq y<n \mbox{ and } y\equiv 1 \hspace{1mm}({\rm mod}\hspace{1mm}d),
\end{equation}where
\begin{equation}
[u]_r=\left\{\barr{lll}
    \!\!1+r+\cdots+r^{u-1},&u>0\\
    \!\!0,&u=0\earr\right..\nonumber
\end{equation}In particular, we have
\begin{equation}
|{\rm Aut}({\rm ZM}(m,n,r))|=m\varphi(m)\frac{n}{d}\,.\nonumber
\end{equation}
\end{theorem}

We are now able to compute the number of automorphism-conjugacy classes of $G$.

\begin{theorem}
The following equality holds:
\begin{equation}
k'({\rm ZM}(m,n,r))=\dd\frac{d}{m\varphi(m)}\dd\sum_{(x_1,x_2,y)\in A}\frac{(m,x_1-1)}{\left[\frac{n}{(n,y-1)}\,, o_{\frac{(m,x_1-1)}{(m,x_1-1,x_2)}}(r)\right]}\,,
\end{equation}where $A$ is the set of triples $(x_1,x_2,y)$ satisfying the conditions (1).
\end{theorem}

\begin{proof}
We will apply Burnside's lemma to the natural action of $${\rm Aut}(G)=\{\varphi_{(x_1,x_2,y)}\mid (x_1,x_2,y)\in A\}$$on $G$. Let $(x_1,x_2,y)\in A$ and
\begin{equation}
Fix(\varphi_{(x_1,x_2,y)})=H_{(m_1,n_1,s)}=\langle a^{m_1},\, b^{n_1}a^s\rangle,\nonumber
\end{equation}where $(m_1,n_1,s)\in L$. Since $H_{(m_1,n_1,s)}$ is the join of its subgroups $\langle a^{m_1}\rangle$ and $\langle b^{n_1}a^s\rangle$, $m_1|m$ and $n_1|n$ are minimal with the property that there is $0\leq s<m_1$ such that $m_1|s\frac{[n]_r}{[n_1]_r}$ and both $a^{m_1}$ and $b^{n_1}a^s$ are fixed points of $\varphi_{(x_1,x_2,y)}$. We have
\begin{equation}
a^{m_1}\!\in\! Fix(\varphi_{(x_1,x_2,y)})\!\Leftrightarrow\! a^{x_1m_1}\!\!=a^{m_1}\!\Leftrightarrow\! m\,|\,m_1(x_1-1)\!\Leftrightarrow\!\frac{m}{(m,x_1-1)}\,|\,m_1,\nonumber
\end{equation}which leads to $m_1=\frac{m}{(m,x_1-1)}$\,, and
\begin{equation}
b^{n_1}a^s\!\in\! Fix(\varphi_{(x_1,x_2,y)})\!\Leftrightarrow\! b^{yn_1}a^{x_1s+x_2[n_1]_r}\!\!=b^{n_1}a^s\!\Leftrightarrow\!\left\{\barr{lll}
\!\!n\,|\,n_1(y-1)\\
\!\!m\,|\,s(x_1-1)+x_2[n_1]_r\earr\right..\nonumber
\end{equation}Thus $n_1$ is minimal such that $\frac{n}{(n,y-1)}\,|\,n_1$ and there is $0\leq s<m_1$ with
\begin{equation}
\left\{\barr{lll}
\!\!m\,|\,s(x_1-1)\frac{[n]_r}{[n_1]_r}\\
\!\!m\,|\,s(x_1-1)+x_2[n_1]_r\earr\right..\nonumber
\end{equation}Since 
\begin{equation}
m\,|\,s(x_1-1)+x_2[n_1]_r\Rightarrow m\,|\,s(x_1-1)\frac{[n]_r}{[n_1]_r}\,,\nonumber
\end{equation}it follows that $n_1$ is minimal such that $\frac{n}{(n,y-1)}\,|\,n_1$ and
\begin{equation}
\!\!m\,|\,s(x_1-1)+x_2[n_1]_r\nonumber
\end{equation}for some $0\leq s<m_1$, i.e. the congruence
\begin{equation}
(x_1-1)s\equiv -x_2[n_1]_r\, ({\rm mod}\, m),\nonumber
\end{equation}has solutions. This is equivalent to
\begin{equation}
(m,x_1-1)\,|\,x_2[n_1]_r,\nonumber
\end{equation}i.e.
\begin{equation}
\frac{(m,x_1-1)}{(m,x_1-1,x_2)}\,|\,r^{n_1}-1,\nonumber
\end{equation}which means that
\begin{equation}
o_{\frac{(m,x_1-1)}{(m,x_1-1,x_2)}}(r)\,|\,n_1.\nonumber
\end{equation}Thus
\begin{equation}
n_1=\left[\frac{n}{(n,y-1)}\,, o_{\frac{(m,x_1-1)}{(m,x_1-1,x_2)}}(r)\right]\nonumber
\end{equation}and
\begin{equation}
|Fix(\varphi_{(x_1,x_2,y)})|=\frac{mn}{m_1n_1}=\frac{n(m,x_1-1)}{\left[\frac{n}{(n,y-1)}\,, o_{\frac{(m,x_1-1)}{(m,x_1-1,x_2)}}(r)\right]}.\nonumber
\end{equation}Then we get
\begin{align*}
k'(G)&=\dd\frac{d}{mn\varphi(m)}\dd\sum_{(x_1,x_2,y)\in A}\frac{n(m,x_1-1)}{\left[\frac{n}{(n,y-1)}\,, o_{\frac{(m,x_1-1)}{(m,x_1-1,x_2)}}(r)\right]}\\
&=\dd\frac{d}{m\varphi(m)}\dd\sum_{(x_1,x_2,y)\in A}\frac{(m,x_1-1)}{\left[\frac{n}{(n,y-1)}\,, o_{\frac{(m,x_1-1)}{(m,x_1-1,x_2)}}(r)\right]}\,,\nonumber
\end{align*}as desired.
\end{proof}

In the notation in the proof of Theorem 2.2, we observe that
\begin{equation}
\frac{n}{(n,y-1)}\leq n_1\leq n,\nonumber
\end{equation}which leads to
\begin{equation}
(m,x_1-1)\leq |Fix(\varphi_{(x_1,x_2,y)})|\leq (m,x_1-1)(n,y-1).\nonumber
\end{equation}By summing over all $(x_1,x_2,y)\in A$ and using the well-known Menon's identity \cite{7}, we get
\begin{equation}
\tau(m)\leq k'(G)\leq\frac{d^2}{n}f\left(\frac{n}{d}\right)\tau(m),\nonumber
\end{equation}where $f$ is the arithmetic function defined by $$f(\alpha)=\sum_{\beta=0}^{\alpha-1}(\alpha,\beta).$$Note that if $\alpha=p_1^{\varepsilon_1}\cdots p_k^{\varepsilon_k}$ is the decomposition of $\alpha$ as a product of prime factors, then $$f(\alpha)=\alpha\prod_{i=1}^k\left(\varepsilon_i+1-\frac{\varepsilon_i}{p_i}\right).$$
\smallskip

A simple formula for $k'(G)$ is obtained if $n$ is a prime number.

\begin{corollary}
If $n$ is a prime number, then the following equality holds:
\begin{equation}
k'({\rm ZM}(m,n,r))=n-1+\tau(m).
\end{equation}
\end{corollary}

\begin{proof}
Since $n$ is prime, we obtain $d=n$ and $y=1$. By (2), it follows that
\begin{equation}
k'(G)=\dd\frac{d}{m\varphi(m)}\dd\sum_{^{0\leq x_1,x_2<m}_{\ (x_1,m)=1}}\frac{(m,x_1-1)}{o_{\frac{(m,x_1-1)}{(m,x_1-1,x_2)}}(r)}\,.\nonumber
\end{equation}Moreover, we have
\begin{equation}
n_1=o_{\frac{(m,x_1-1)}{(m,x_1-1,x_2)}}(r)=\left\{\barr{lll}
\!\!1, &\mbox{  if }\, (m,x_1-1)\mid x_2\\
\!\!n, &\mbox{  if }\, (m,x_1-1)\nmid x_2\earr\right.\nonumber
\end{equation}and so
\begin{equation}
|Fix(\varphi_{(x_1,x_2,y)})|=\left\{\barr{lll}
\!\!n(m,x_1-1), &\mbox{  if }\, (m,x_1-1)\mid x_2\\
\!\!(m,x_1-1), &\mbox{  if }\, (m,x_1-1)\nmid x_2\earr\right.\nonumber
\end{equation}Then
\begin{align*}
k'(G)&=\dd\frac{1}{m\varphi(m)}\!\!\dd\sum_{^{\ 0\leq x_1<m}_{\ (x_1,m)=1}}\!\!\left(\,\dd\sum_{0\leq x_2<m}|Fix(\varphi_{(x_1,x_2,y)})|\right)
\end{align*}
\begin{align*}
&=\dd\frac{1}{m\varphi(m)}\!\!\dd\sum_{^{\ 0\leq x_1<m}_{\ (x_1,m)=1}}\!\!\left(\,\dd\sum_{^{\ 0\leq x_2<m}_{(m,x_1-1)\mid x_2}}\!\!\!n(m,x_1-1)+\!\!\!\!\!\!\dd\sum_{^{\ 0\leq x_2<m}_{(m,x_1-1)\nmid x_2}}\!\!\!(m,x_1-1)\right)\\
&=\dd\frac{1}{m\varphi(m)}\!\!\dd\sum_{^{\ 0\leq x_1<m}_{\ (x_1,m)=1}}\!\!\left[mn+(m,x_1-1)\left(m-\frac{m}{(m,x_1-1)}\right)\right]\\
&=\dd\frac{1}{m\varphi(m)}\left(mn\varphi(m)+m\!\!\!\dd\sum_{^{\ 0\leq x_1<m}_{\ (x_1,m)=1}}(m,x_1-1)-m\varphi(m)\right)\\
&=\dd\frac{1}{m\varphi(m)}\left(mn\varphi(m)+m\varphi(m)\tau(m)-m\varphi(m)\right)\\
&=n-1+\tau(m),
\end{align*}completing the proof.
\end{proof}

In particular, for $n=2$ and $r=m-1$ we have ${\rm ZM}(m,n,r)\cong D_{2m}$, the dihedral group of order $2m$. Thus, the equality (3) becomes
\begin{equation}
k'(D_{2m})=\tau(m)+1, \mbox{ for all odd integers } m\geq 3.
\end{equation}

An example of applying the formula (2) is presented in the following.

\begin{example}
For the dicyclic group ${\rm Dic}_3={\rm ZM}(3,4,2)={\rm SmallGroup}(12,1)$, we have $d=2$ and (2) leads to
\begin{equation}
k'({\rm Dic}_3)=5.\nonumber
\end{equation}
\end{example}

Next, we will focus on computing the number of conjugacy classes of $G$.

\begin{theorem}
The following equality holds:
\begin{equation}
k({\rm ZM}(m,n,r))=\dd\frac{n}{md}\dd\sum_{\alpha=0}^{d-1}\dd\sum_{\beta=0}^{m-1}\frac{(m,r^{\alpha}-1)}{o_{\frac{(m,r^{\alpha}-1)}{(m,r^{\alpha}-1,\beta)}}(r)}\,.
\end{equation}
\end{theorem}

\begin{proof}
We have
\begin{equation}
{\rm Inn}(G)=\{\varphi_{b^{\alpha}a^{\beta}}\mid \alpha=\overline{0,d-1}, \beta=\overline{0,m-1}\},\nonumber
\end{equation}where $\varphi_{b^{\alpha}a^{\beta}}$ is the conjugation by $b^{\alpha}a^{\beta}$. In the notation in Theorem 2.2, it is easy to see that
\begin{equation}
\varphi_{b^{\alpha}a^{\beta}}=\varphi_{(r^{\alpha},\beta(1-r),1)}\nonumber
\end{equation}and so
\begin{equation}
|Fix(\varphi_{b^{\alpha}a^{\beta}})|=\frac{n(m,r^{\alpha}-1)}{o_{\frac{(m,r^{\alpha}-1)}{(m,r^{\alpha}-1,\beta)}}(r)}\,.\nonumber
\end{equation}Then (5) follows by applying Burnside's lemma to the natural action of ${\rm Inn}(G)$ on $G$.
\end{proof}

We remark that 
\begin{equation}
o_{\frac{(m,r^{\alpha}-1)}{(m,r^{\alpha}-1,\beta)}}(r)=1\Leftrightarrow\frac{(m,r^{\alpha}-1)}{(m,r^{\alpha}-1,\beta)}=1\Leftrightarrow (m,r^{\alpha}-1)\mid\beta.
\end{equation}Also, for $(m,r^{\alpha}-1)\nmid\beta$ we have
\begin{equation}
\frac{1}{d}\leq \frac{1}{o_{\frac{(m,r^{\alpha}-1)}{(m,r^{\alpha}-1,\beta)}}(r)}\leq\frac{1}{p}\,,
\end{equation}where $p$ is the smallest prime divisor of $d$. Then, writing 
\begin{equation}
k(G)=\dd\frac{n}{md}\dd\sum_{\alpha=0}^{d-1}(m,r^{\alpha}-1)\!\!\left(\dd\sum_{^{0\leq\beta<m-1}_{(m,r^{\alpha}-1)\mid\beta}}\!\!\frac{1}{o_{\frac{(m,r^{\alpha}-1)}{(m,r^{\alpha}-1,\beta)}}(r)}+\!\!\!\!\dd\sum_{^{0\leq\beta<m-1}_{(m,r^{\alpha}-1)\nmid\beta}}\!\!\frac{1}{o_{\frac{(m,r^{\alpha}-1)}{(m,r^{\alpha}-1,\beta)}}(r)}\!\right)\nonumber
\end{equation}and using (6) and (7), we get
\begin{equation}
\frac{n}{d}\left(d-1+\frac{S}{d}\right)\leq k(G)\leq\frac{n}{d}\left(d-\frac{d}{p}+\frac{S}{p}\right),
\end{equation}where $$S=\sum_{\alpha=0}^{d-1}(m,r^{\alpha}-1).$$Note that if $d$ is prime, then $p=d$ and (8) leads to the equality
\begin{equation}
k(G)=\frac{n}{d}\left(d-1+\frac{S}{d}\right).
\end{equation}By considering the natural action of $\mathbb{Z}^{\times}_m$ on $\mathbb{Z}_m$, we have $$S=\sum_{\alpha=0}^{d-1}|Fix(\hat{r}^{\,\alpha})|,$$implying that
\begin{equation}
m+d-1\leq S\leq\sum_{\hat{a}\in\mathbb{Z}^{\times}_m}|Fix(\hat{a})|=\varphi(m)\tau(m).\nonumber
\end{equation}These inequalities, together with (8), show that
\begin{equation}
\frac{n(m+d^2-1)}{d^2}\leq k(G)\leq\frac{n[\varphi(m)\tau(m)+d(p-1)]}{dp}\,.\nonumber
\end{equation}
\smallskip

Again, if $n$ is prime, then $d=n$ is also prime and (9) gives a simple formula for $k(G)$.

\begin{corollary}
If $n$ is a prime number, then the following equality holds:
\begin{equation}
k({\rm ZM}(m,n,r))=n-1+\frac{1}{n}\sum_{\alpha=0}^{n-1}(m,r^{\alpha}-1).\nonumber
\end{equation}
\end{corollary}

In particular, for $n=2$ and $r=m-1$ we get
\begin{equation}
k(D_{2m})=\frac{m+3}{2}, \mbox{ for all odd integers } m\geq 3.\nonumber
\end{equation}Note that for even integers $m\geq 2$ this equality becomes
\begin{equation}
k(D_{2m})=\frac{m+6}{2}\,.\nonumber
\end{equation}

The equality (9) can be also used to compute the number of conjugacy classes of ${\rm Dic}_3$, for which we obtain:

\begin{example}
$k({\rm Dic}_3)=6$.\nonumber
\end{example}

Finally, we compute the sizes $|O'_{b^ua^v}|$ and $|O_{b^ua^v}|$ of automorphism-conjugacy class and conjugacy class of an arbitrary element $b^ua^v\in G$.

\begin{theorem}
The following equalities hold:
\begin{itemize}
\item[{\rm a)}] $|O'_{b^ua^v}|=\dd\frac{m\,n\,\varphi\!\left(\dd\frac{(m,[u]_r)}{(m,[u]_r,v)}\right)}{(n,u)\left(\dd\frac{n}{(n,u)}\,,d\right)(m,[u]_r)}$\,;
\item[{\rm b)}] $|O_{b^ua^v}|=\dd\frac{m}{(m,[u]_r)}\,o_{\frac{(m,[u]_r)}{(m,[u]_r,v)}}(r)$.
\end{itemize}
\end{theorem}

\begin{proof}
In both cases we will apply the Orbit-Stabilizer theorem. Also, we will use the notation
\begin{equation}
g=\dd\frac{(m,[u]_r)}{(m,[u]_r,v)}\,.\nonumber
\end{equation}
\begin{itemize}
\item[{\rm a)}] We have
\begin{equation}
|O'_{b^ua^v}|=[{\rm Aut}(G): {\rm Stab}_{{\rm Aut}(G)}(b^ua^v)].\nonumber
\end{equation}An automorphism $\varphi_{(x_1,x_2,y)}$ is contained in ${\rm Stab}_{{\rm Aut}(G)}(b^ua^v)$ if and only if $b^ua^v\in Fix(\varphi_{(x_1,x_2,y)})$, that is if and only if
\begin{equation}
\left\{\barr{lll}
\!\!n\,|\,u(y-1)\\
\!\!m\,|\,v(x_1-1)+x_2[u]_r.\earr\right.\nonumber
\end{equation}Put $y=zd+1$, where $z\in\{0,1,...,\frac{n}{d}-1\}$\,. Then the first condition is equivalent with $\frac{n}{(n,u)}\mid zd$, i.e.\vspace{-2mm} \begin{equation}
\dd\frac{n}{h}\mid z,\vspace{-2mm}\nonumber 
\end{equation}where $h=(n,u)\left(\frac{n}{(n,u)}\,,d\right)$, and so $y$ can be chosen in $\frac{n}{d}:\frac{n}{h}=\frac{h}{d}$ ways. We also have
\begin{align*}
|\{(x_1,x_2)&\in\overline{0,m-1}^{\,2}\mid (x_1,m)=1, m\,|\,v(x_1-1)+x_2[u]_r\}|\\
&=\dd\sum_{^{\,\,\,\,0\leq x_1<m}_{\ (x_1,m)=1}}|\{x_2\in\overline{0,m-1}\mid m\,|\,v(x_1-1)+x_2[u]_r\}|\\
&=\!\!\!\dd\sum_{^{\,\,\,\,\,\,\,\,\,\,\,\,\,\,0\leq x_1<m}_{\ (x_1,m)=1,\, g\,|\,x_1-1}}\!\!\!\!\!(m,[u]_r)\\
&=(m,[u]_r)\,\dd\frac{\varphi(m)}{\varphi(g)}\,.
\end{align*}Thus
\begin{align*}
|O'_{b^ua^v}|&=\dd\frac{|{\rm Aut}(G)|}{|{\rm Stab}_{{\rm Aut}(G)}(b^ua^v)|}\\
&=m\varphi(m)\frac{n}{d}: \dd\frac{h}{d}\,(m,[u]_r)\,\dd\frac{\varphi(m)}{\varphi(g)}\\
&=\dd\frac{m\,n\,\varphi(g)}{h(m,[u]_r)}\,.
\end{align*}
\item[{\rm b)}] We have\vspace{-1mm}
\begin{equation}
|O_{b^ua^v}|=[{\rm Inn}(G): {\rm Stab}_{{\rm Inn}(G)}(b^ua^v)]\vspace{-1mm}\nonumber
\end{equation}and
\begin{align*}
|{\rm Stab}_{{\rm Inn}(G)}(b^ua^v)|&=|\{(\alpha,\beta)\in\overline{0,d-1}\times\overline{0,m-1}\mid b^ua^v\in Fix(\varphi_{(r^{\alpha},\beta(1-r),1)})\}|\\
&=\dd\sum_{0\leq\alpha<d}|\{\beta=\overline{0,m-1}\mid m\,|\,v(r^{\alpha}-1)+\beta(1-r)[u]_r\}|\\
&=\dd\sum_{^{\,\,\,\,\,\,\,\,\,\,\,\,\,0\leq\alpha<d}_{(m,[u]_r)|v(r^{\alpha}-1)}}\!\!\!(m,[u]_r)\\
&=\dd\sum_{^{\,\,\,0\leq\alpha<d}_{\,\,\,g|r^{\alpha}-1}}\!(m,[u]_r)\\
&=\dd\sum_{^{\,\,\,0\leq\alpha<d}_{\,\,\,o_{g}(r)|\alpha}}\!(m,[u]_r)\\
&=\dd\frac{d\,(m,[u]_r)}{o_{g}(r)}\,.
\end{align*}Thus
\begin{align*}
|O_{b^ua^v}|&=\dd\frac{|{\rm Inn}(G)|}{|{\rm Stab}_{{\rm Inn}(G)}(b^ua^v)|}\\
&=\dd\frac{dm}{\dd\frac{d\,(m,[u]_r)}{o_{g}(r)}}\\
&=\dd\frac{m}{(m,[u]_r)}\,o_{g}(r)
\end{align*}
\end{itemize}and the proof of our theorem is complete.
\end{proof}

Since ${\rm Stab}_{{\rm Inn}(G)}(b^ua^v)\cong C_{G}(b^ua^v)$, the above proof also gives the order of the centralizer of $b^ua^v$ in $G$.
\newpage

\noindent{\bf Acknowledgements.} The author is grateful to the reviewer for remarks which improve the previous version of the paper.

\vspace*{5ex}\small

\hfill
\begin{minipage}[t]{5cm}
Marius T\u arn\u auceanu \\
Faculty of  Mathematics \\
``Al.I. Cuza'' University \\
Ia\c si, Romania \\
e-mail: {\tt tarnauc@uaic.ro}
\end{minipage}

\end{document}